\documentclass{amsart}
\usepackage{bbding}
\usepackage{amssymb}
\usepackage{amsmath}

\newtheorem{theorem}{Theorem}[section]
\newtheorem{lemma}[theorem]{Lemma}
\newtheorem{proposition}[theorem]{Proposition}
\newtheorem{corollary}[theorem]{Corollary}
\theoremstyle{definition}
\newtheorem{definition}[theorem]{Definition}
\newtheorem{example}[theorem]{Example}

\newtheorem{remark}[theorem]{Remark}
\numberwithin{equation}{section}

\newcommand{\blankbox}[2]

\begin{document}
\title{Simplicity of quadratic Lie conformal algebras}

\author{Yanyong Hong}
\address{College of Science, Zhejiang Agriculture and Forestry University,
Hangzhou, 311300, P.R.China}
\email{hongyanyong2008@yahoo.com}
\author{Zhixiang Wu}
\address{Department of Mathematics, Zhejiang University,
Hangzhou, 310027, P.R.China}
\email{wzx@zju.edu.cn}

\thanks{Project supported by the National Natural Science Foundation of China (No. 11171296 and No.11401533), the Zhejiang Provincial Natural Science Foundation of China (No. LZ14A010001) and the Scientific Research Foundation of Zhejiang Agriculture and Forestry University (No.2013FR081)}
\begin{abstract}
In this paper, simplicity of quadratic Lie conformal algebras are investigated. From the point view of the corresponding Gel'fand-Dorfman bialgebras, some sufficient conditions and necessary conditions to ensure simplicity of quadratic Lie conformal algebras are presented. By these
observations, we present several classes of new infinite simple Lie conformal algebras. These results will be useful for classification purposes.
\end{abstract}

\maketitle

\section{Introduction}
The vertex algebra introduced by R.Borcherds
in \cite{Bo} is a rigorous mathematical concept of the
chiral part of a 2-dimensional quantum field theory. It has been studied
intensively by physicists since the landmark paper \cite{BPZ} appeared. The notation of Lie
conformal algebra was formulated by Kac in \cite{K1,K2}. It is an
axiomatic description of the operator product expansion (or rather
its Fourier transform) of chiral fields in conformal field theory.
Lie conformal algebras play important roles in quantum field theory, vertex algebras and infinite-dimensional Lie algebras satisfying the locality property in \cite{K}.
Moreover, Lie conformal algebras have close connections to Hamiltonian formalism in the theory of nonlinear evolution equations (see the book \cite{Do} and references therein, and also \cite{BDK, GD, Z, X4} and many other papers).

By now, there are only two classes of Lie conformal algebras, namely, the Lie conformal algebra $gc_N$ and its infinite subalgebras, and quadratic Lie conformal algebras named by Xu in \cite{X1}. A quadratic Lie conformal algebra corresponds to a Hamiltonian pair in \cite{GD}, which plays fundamental roles in completely integrable systems. Moreover, it is completely determined by a Gel'fand-Dorfman bialgebra \cite{X1}. Central extensions and conformal derivations of quadratic Lie conformal algebras are studied by us in \cite{H} in terms of Gel'fand-Dorfman bialgebras.

A Lie conformal algebra  is said to be finite if it is a finitely generated  $\mathbb{C}[\partial]$-module. Otherwise, it is called an infinite Lie conformal algebra.
The structure theory, representation theory and cohomology theory of finite Lie conformal algebras have been developed in recent years (e.g.,\cite{DK1,CK1,CK2,BKV}). A complete classification of finite simple (or semisimple) Lie conformal algebras is given in \cite{DK1},
all irreducible finite conformal modules of finite simple (or semisimple) Lie conformal algebras are classified in \cite{CK1, CK2}
and cohomology groups of finite simple Lie conformal algebras with some conformal modules
are characterized in \cite{BKV}. However, there is a little progression on the study of infinite Lie conformal algebras
as far as we know.  The intensive studying infinite Lie conformal algebra is the general Lie conformal algebra $gc_N$ \cite{BKL1,BKL2,DK2,S,SY1}.
$gc_N$ plays the same important role in the theory of Lie conformal algebras as the general Lie algebras $gl_N$ does in the theory of Lie algebras.
In addition, some infinite Lie conformal algebras obtained from some known formal distribution Lie algebras are studied (e.g.,\cite{FCH,WCY}).

The fundamental question in the studying infinite Lie conformal algebras is to classify the infinite simple Lie conformal algebras. To solve this question, we must construct some infinite simple Lie conformal algebras. This is the motivation of this paper. First, we study simplicity of quadratic Lie conformal algebras through the corresponding Gel'fand-Dorfman bialgebras.  Some sufficient conditions and necessary conditions to ensure simplicity
of quadratic Lie conformal algebras are presented. Then we construct several classes of new infinite simple Lie conformal algebras by using Gel'fand-Dorfman bialgebras. This will enrich the theory of infinite Lie conformal algebras and will be useful for classification purposes.

This paper is organized as follows. In Section 2, the definitions of Lie conformal algebra and quadratic Lie conformal algebra are recalled.
In Section 3, some sufficient conditions and necessary conditions to ensure simplicity of quadratic Lie conformal algebras are presented and we give a strategy to find infinite simple Lie conformal algebras. In Section 4, several classes of new infinite simple Lie conformal algebras are constructed.

Throughout this paper, denote by $\mathbb{C}$ the field of complex
numbers; $\mathbb{C}^+$ is the additive group of $\mathbb{C}$; $\mathbb{N}$ is the set of natural numbers, i.e.
$\mathbb{N}=\{0, 1, 2,\cdots\}$; $\mathbb{Z}$ is the set of integer numbers; $\mathbb{Z}_+$ is the set of positive integer numbers.

\section{Preliminaries}
In this section, some results about Lie conformal algebras and quadratic Lie conformal algebras are recalled and we refer to \cite{K1} and \cite{X1}.
\begin{definition}
A \emph{Lie conformal algebra} $R$ is a $\mathbb{C}[\partial]$-module with a $\lambda$-bracket $[\cdot_\lambda \cdot]$ which defines a $\mathbb{C}$-bilinear
map from $R\otimes R\rightarrow R[\lambda]$, where $R[\lambda]= R\otimes\mathbb{C}[\lambda]$ is the space of polynomials of $\lambda$ with coefficients
in $R$, satisfying
\begin{eqnarray*}
&&[\partial a_\lambda b]=-\lambda [a_\lambda b],~~~[a_\lambda \partial b]=(\lambda+\partial)[a_\lambda b], ~~\text{(conformal sesquilinearity)}\\
&&[a_\lambda b]=-[b_{-\lambda-\partial}a],~~~~\text{(skew-symmetry)}\\
&&[a_\lambda[b_\mu c]]=[[a_\lambda b]_{\lambda+\mu} c]+[b_\mu[a_\lambda c]],~~~~~~\text{(Jacobi identity)}
\end{eqnarray*}
for $a$, $b$, $c\in R$.
\end{definition}

If a Lie conformal algebra $R$ is a finitely generated
$\mathbb{C}[\partial]$-module, it is called a \emph{finite} Lie conformal algebra; otherwise, it is said to be  \emph{infinite}.

 Suppose $A$ and $B$ are subspaces of a Lie conformal algebra $R$. Using the notations in \cite{D}, we define $[A,B]$ as the $\mathbb{C}$-linear span of all $\lambda$-coefficients in the products $[a_\lambda b]$, where $a\in A$, $b\in B$. Notice that if $A$ and $B$ are both $\mathbb{C}[\partial]$-submodules, then $[A, B]=[B, A]$ by skew symmetry. An \emph{ideal} of a Lie conformal algebra $R$ is a $\mathbb{C}[\partial]$-submodule $I\subset R$ such that $[R,I]\subset I$. Every nonzero Lie conformal algebra has two ideals, namely $0$ and itself. These two ideals are called \emph{trivial}. $R$ is \emph{abelian} if $[R,R]=0$. A Lie conformal algebra $R$ is \emph{simple} if its only ideals are trivial and it is not abelian.

Moreover, there is an important infinite-dimensional Lie algebra associated with a Lie conformal algebra.
Assume that $R$ is a Lie conformal algebra and set $[a_\lambda b]=\sum\limits_{n\in \mathbb{N}}\frac{\lambda^n}{n!}a_{(n)}b$. Let Coeff$(R)$ be the quotient
of the vector space with basis $a_n$ $(a\in R, n\in\mathbb{Z})$ by
the subspace spanned over $\mathbb{C}$ by
elements:
$$(\alpha a)_n-\alpha a_n,~~(a+b)_n-a_n-b_n,~~(\partial
a)_n+na_{n-1},~~~\text{where}~~a,~~b\in R,~~\alpha\in \mathbb{C},~~n\in
\mathbb{Z}.$$ The operation on Coeff$(R)$ is given as follows:
\begin{equation}\label{106}
[a_m, b_n]=\sum_{j\in \mathbf{N}}\left(\begin{array}{ccc}
m\\j\end{array}\right)(a_{(j)}b)_{m+n-j}.\end{equation} Then,
Coeff$(R)$ is a Lie algebra and it is called the\emph{ coefficient algebra} of $R$ (see \cite{K1}).

Next, we introduce some examples of Lie conformal algebras.

\begin{example}

Let $\mathfrak{g}$ be a Lie algebra. The current Lie conformal
algebra associated to $\mathfrak{g}$ is defined by:
$$\text{Cur} \mathfrak{g}=\mathbb{C}[\partial]\otimes \mathfrak{g}, ~~[a_\lambda b]=[a,b],
~~a,b\in \mathfrak{g}.$$
\end{example}

\begin{example}
The Virasoro Lie conformal algebra $\text{Vir}$ is the simplest nontrivial
example of Lie conformal algebras. It is defined by
$$\text{Vir}=\mathbb{C}[\partial]L, ~~[L_\lambda L]=(\partial+2\lambda)L.$$
Coeff$\text{(Vir)}$ is just the Witt algebra.
\end{example}

It is shown in \cite{DK1} that any finite simple
Lie conformal algebra is isomorphic to either $\text{Vir}$ or $\text{Cur} \mathfrak{g}$, where $\mathfrak{g}$ is a
finite-dimensional simple Lie algebra.

\begin{definition}
 If $R=\mathbb{C}[\partial]V$ is a Lie conformal algebra as a free
$\mathbb{C}[\partial]$-module and the $\lambda$-bracket is of the following form:
\begin{eqnarray*}
[a_{\lambda} b]=\partial u+\lambda v+ w,~~~~~~\text{$a$, $b\in V$,}
\end{eqnarray*}
where $u$, $v$, $w\in V$, then $R$ is called a \emph{quadratic Lie conformal algebra}.
\end{definition}

Obviously, $\text{Vir}$ and $\text{Cur}\mathfrak{g}$ for a Lie algebra $\mathfrak{g}$ are quadratic Lie conformal algebras.

\begin{definition}
A \emph{Novikov algebra} $V$ is a vector space over $\mathbb{C}$ with a bilinear product $\circ: V\times V\rightarrow V$ satisfying (for any $a$, $b$, $c\in V$):
\begin{eqnarray}
&&(a\circ b)\circ c-a \circ (b\circ c)=(b\circ a)\circ c-b \circ (a\circ c),\\
&&(a\circ b)\circ c=(a\circ c)\circ b.
\end{eqnarray}
\end{definition}

If $I$ is a subspace of a Novikov algebra $(V,\circ)$ and $V\circ I\subset I$, $I\circ V\subset I$, then $I$ is called an \emph{ideal} of $(V,\circ)$. Obviously, $0$ and $V$ are ideals of $(V,\circ)$, which are called \emph{trivial}. $(V,\circ)$ is called \emph{simple}, if
$(V,\circ)$ has only trivial ideals and $a\circ b\neq 0$ for some $a$, $b\in V$.

\begin{remark}
Novikov algebra was essentially stated in \cite{GD}. It corresponds to a
certain Hamiltonian operator and also appeared
in \cite{BN} from the point of view of Poisson structures of
hydrodynamic type. The name ``Novikov algebra" was given by Osborn
in \cite{Os}.
\end{remark}

\begin{definition}(see \cite{GD} or \cite{X1})
A \emph{Gel'fand-Dorfman bialgebra} $V$ is a Lie algebra $(V,[\cdot,\cdot])$ with a binary operation $\circ$ such that $(V,\circ)$ forms a Novikov algebra and the following compatibility condition holds:
\begin{eqnarray}\label{eqq3}
[a\circ b, c]-[a\circ c, b]+[a,b]\circ c-[a,c]\circ b-a\circ [b,c]=0,
\end{eqnarray}
for $a$, $b$, and $c\in V$. We usually denote it by $(V,\circ,[\cdot,\cdot])$.
\end{definition}
Obviously, every Lie algebra $L$ is a Gel'fand-Dorfman bialgebra with the trivial Novikov algebra structure, namely $a\circ b=0$ for any $a,b\in L$. Similarly,
any Novikov algebra $N$ is a Gel'fand-Dorfman bialgebra with the trivial Lie bracket $[a,b]=0$ for any $a,b\in N$.

Moreover, there is a natural construction of Gel'fand-Dorfman bialgebras from Novikov algebras.

\begin{proposition}\label{pp2}
Let $(V,\circ)$ be a Novikov algebra. Define a Lie bracket $[\cdot,\cdot]^-_k$ on $V$ as follows:
\begin{eqnarray}
[a,b]^-_k=k(a\circ b-b\circ a), \text{for any $a$, $b\in V$, $k\in \mathbb{C}$.}
\end{eqnarray}
Then, $(V,\circ,[\cdot,\cdot]^-_k)$ forms a Gel'fand-Dorfman bialgebra.
\end{proposition}

\begin{proof}
It can be referred to Theorem 2.3 in \cite{X1}.
\end{proof}

Let $(V, \circ, [\cdot,\cdot])$ be a Gel'fand-Dorfman  bialgebra. For convenience, we call $(V,[\cdot,\cdot])$ a \emph{Lie algebra over the Novikov algebra} $(V,\circ)$ and $(V,\circ)$ a \emph{Novikov algebra over the Lie algebra} $(V,[\cdot,\cdot])$.

A  subspace $I$ of a Gel'fand-Dorfman bialgebra $(V, \circ, [\cdot,\cdot])$ is called a \emph{Gel'fand-Dorfman ideal} if it is an ideal of the Lie algebra $(V,[\cdot,\cdot])$  and an ideal of Novikov algebra $(V, \circ)$. The zero ideal and $V$ itself are trivial Gel'fand-Dorfman ideals of any Gel'fand-Dorfman bialgebra $(V, \circ, [\cdot,\cdot])$. The other ideals are said to be \emph{proper}. Obviously, for a Gel'fand-Dorfman bialgebra $(V, \circ, [\cdot,\cdot])$, if $(V,\circ)$ or $(V,[\cdot,\cdot])$ is simple, then $(V, \circ, [\cdot,\cdot])$ has no proper Gel'fand-Dorfman ideals.

An equivalent characterization of quadratic Lie conformal algebras is given as follows.
\begin{theorem}(see \cite{GD} or \cite{X1})
$R=\mathbb{C}[\partial]V$ is a quadratic Lie conformal algebra if and only if the $\lambda$-bracket of $R$ is given as
follows
$$[a_{\lambda} b]=\partial(b\circ a)+[b, a]+\lambda(b\circ a+a\circ b), \text{$a$, $b\in V$},$$
and $(V, \circ, [\cdot,\cdot])$ is a Gel'fand-Dorfman bialgebra. Therefore, $R$ is called the quadratic Lie conformal algebra corresponding to
the Gel'fand-Dorfman bialgebra $(V, \circ, [\cdot,\cdot])$.
\end{theorem}

Moreover, if a Gel'fand-Dorfman bialgebra is a Novikov algebra with a trivial Lie algebra structure, for convenience, we usually say the \emph{quadratic Lie conformal algebra corresponds to the Novikov algebra.}

\section{Simplicity of quadratic Lie conformal algebras}

In this section, we will give some sufficient conditions and necessary conditions for the simplicity of quadratic Lie conformal algebras. Since the current Lie conformal algebra $R=\mathbb{C}[\partial]\mathfrak{g}$ associated with a Lie algebra $\mathfrak{g}$ is simple
if and only if $\mathfrak{g}$ is simple, we always assume that quadratic Lie conformal algebras are not current in the sequel, i.e., the Novikov algebra structure of the corresponding Gel'fand-Dorfman bialgebra is non-trivial.

To achieve our purpose, let us define $a\ast b=a\circ b+b\circ a$ for any $a$, $b\in V$, where $V$ is a Gel'fand-Dorfman  bialgebra. Then, $(V, \ast)$ is a commutative but not (usually) associative algebra. In fact,
$\ast$ satisfies the following equality:
\begin{eqnarray}\label{eqn}
(a\ast b)\ast(c\ast d)-(a\ast d)\ast(c\ast d)=(a,b,c)\ast d-(a,d,c)\ast b,
\end{eqnarray}
where $(a,b,c)=a\ast(b\ast c)-(a\ast b)\ast c$, and $a$, $b$, $c$, $d\in V$. In \cite{Dz}, (\ref{eqn}) is called \emph{Tortken identity}, and
$(V,\ast)$ is called \emph{Novikov-Jordan algebra}. A subspace $I$ of a Novikov-Jordan algebra $(V,\ast)$ is called an \emph{ideal} of $V$ if $a\ast b\in I$ for
any $a\in I$ and $b\in V$. Any nonzero Novikov-Jordan algebra $V$ has two trivial ideals $0$ and $V$. An algebra $(V,\ast)$ is called a \emph{simple} Novikov-Jordan algebra if $V$ has only trivial ideals and $a\ast b\neq 0$ for some $a,b\in V$. We also denote the $\mathbb{C}$-span of all elements $a\ast b$ by $V\ast V$ where $a$, $b\in V$.

Using the operation $\ast$, the $\lambda$-bracket of a quadratic Lie conformal algebra is given as follows
\begin{eqnarray}\label{gf1}
[a_{\lambda} b]=\partial(b\circ a)+[b, a]+\lambda(a\ast b), \text{$a$, $b\in V$}.
\end{eqnarray}

Then, we present some necessary conditions for the simplicity of quadratic Lie conformal algebras.

\begin{proposition}\label{pp1}
If the quadratic Lie conformal algebra $R=\mathbb{C}[\partial]V$ corresponding to a Gel'fand-Dorfman bialgebra
$(V,\circ, [\cdot,\cdot])$ is simple, then $(V,\circ, [\cdot,\cdot])$ has no proper Gel'fand-Dorfman ideals.

\end{proposition}
\begin{proof}
Suppose that $I$ is a proper Gel'fand-Dorfman ideal of $(V,\circ,[\cdot,\cdot])$. Since $I$ is an ideal of $(V,\circ)$, $I$ is also an ideal of $(V,\ast)$. Then, by (\ref{gf1}), it is easy to see that $\mathbb{C}[\partial]I$ is a non-trivial ideal of $R$ which contradicts with the simplicity of $R$. Thus, this proposition holds.
\end{proof}

\begin{remark} It should be pointed out that the necessary condition in Proposition \ref{pp1} is not sufficient. For example, there are non-trivial Novikov algebra structures over the Lie algebra $sl_2$ (see \cite{X5}). It is obvious that these Gel'fand-Dorfman bialgebras have no proper Gel'fand-Dorfman ideals. Since the Lie conformal algebras corresponding to these Gel'fand-Dorfman bialgebras do not belong to those simple objects classified in \cite{DK1} , they are not simple. \end{remark}

\begin{corollary}\label{cv}
If a simple Lie conformal algebra $R=\mathbb{C}[\partial]V$ corresponds to a Gel'fand-Dorfman bialgebra $(V,\circ,[\cdot,\cdot]^-_k)$ defined in Proposition \ref{pp2}, then $(V,\circ)$ is a simple Novikov algebra.

In special, when $k=0$, the Lie conformal algebra $R=\mathbb{C}[\partial]V$ corresponding to a Novikov algebra $(V,\circ)$ is simple, then $(V,\circ)$ is simple.
\end{corollary}
\begin{proof}
Assume that $(V,\circ)$ is not simple. Then, there exists a non-trivial ideal $I$ of $(V,\circ)$. Moreover,
$I$ is not only a non-trivial ideal of $(V,\ast)$, but also of $(V,[\cdot,\cdot]^-_k)$. Therefore, $I$ is a proper Gel'fand-Dorfman ideal.
By Proposition \ref{pp1}, the corresponding Lie conformal algebra $R=\mathbb{C}[\partial]V$ is not simple. It contradicts with the assumption.
Thus, $(V,\circ)$ is simple.

The second claim can be directly obtained from the first claim.
\end{proof}

Next, we begin to study some sufficient conditions for the simplicity of quadratic Lie conformal algebras.

\begin{lemma}\label{ty}
Let $(V,\circ)$ be a simple Novikov algebra, and $(V,\ast)$ be the corresponding  Novikov-Jordan algebra with the operation $a\ast b=a\circ b+b\circ a$ for $a,b\in V$. Then there is no any non-zero element $a\in V$ such that $a\ast b=0$ for any $b\in V$.
\end{lemma}
\begin{proof}Conversely, if there exists some nonzero element $a\in V$ such that $a\ast b=0$ for any $b\in V$, then $a\circ b=-b\circ a$. Therefore, $a\circ a=0$.
Then, by commutativity of right products, we have \begin{eqnarray}\label{tt}\qquad
(a\circ a)\circ c=(a\circ c)\circ a=-(c\circ a)\circ a=-a\circ(a\circ c)=a\circ(c\circ a)=0.
\end{eqnarray}
Since $(V,\circ)$ is simple, $V$ is spanned by
$a\circ c$ and $c\circ a$ for all $c\in V$. Thus $a\circ b=b\circ a=0$ for any $b\in V$ by (\ref{tt}). Consequently, $a=0$. It contradicts with the assumption.
\end{proof}

\begin{theorem}\label{p1}
Given a Gel'fand-Dorfman bialgebra $(V,\circ, [\cdot,\cdot])$. If $(V,\circ)$ is a simple Novikov algebra, and $V=V\ast V$, then the quadratic Lie conformal algebra $R=\mathbb{C}[\partial]V$ corresponding to $(V,\circ, [\cdot,\cdot])$ is simple.
\end{theorem}
\begin{proof}
Let $I$ be a non-zero ideal of $R$ and $P=\sum\limits_{i=0}^nP_i(\partial)a_i \in I$, where $a_i \in V$ ($0\leq i\leq n$)  are linearly independent and $P_i(\partial)\in \mathbb{C}[\partial]\setminus\{0\}$. Assume that the degrees of $P_m(\partial)$, $\cdots$, $P_n(\partial)$ are maximal in those $P_i(\partial)$, and the leading coefficients of $P_m(\partial)$, $\cdots$, $P_n(\partial)$ are $k_m$, $\cdots$, $k_n$. Let the degree of $P_m(\partial)$ be $t$. For any $a\in V$,
\begin{eqnarray}\label{eq1}
[a_\lambda P]=\sum_{i=0}^n P_i(\lambda+\partial)(\partial(a_i\circ a)+\lambda(a\ast a_i)+[a_i,a]).
\end{eqnarray}
According to the coefficient of $\lambda^{t+1}$ in (\ref{eq1}), we obtain $w=a\ast(k_m a_m+\cdots+k_n a_n)\in I$ for any $a\in V$.
By Lemma \ref{ty}, there exists some $a\in V$ such that $w\neq 0$. Then, since
\begin{eqnarray*}\label{eqr}
[b_\lambda w]=\partial(w\circ b)+\lambda(b\ast w)+[w,b], \text{and} [w_\lambda b]=\partial(b\circ w)+\lambda(b\ast w)+[b,w],
\end{eqnarray*}
we get $\partial(w\circ b)+[w,b]$, $\partial(b\circ w)+[b,w]\in I$ from the coefficient of $\lambda^0$ for any $b\in V$.
Since $(V,\circ)$ is simple, $\partial u+v\in I$ for any $u\in V$, and some $v\in V$.
By \begin{eqnarray*}
[b_\lambda \partial u+v]&&=(\lambda+\partial)(\partial(u\circ b)+\lambda(b\ast u)+[u,b])\\
&&+\partial(v\circ b)+\lambda(b\ast v)+[v,b],
\end{eqnarray*}
we obtain $b\ast u \in I$ from the coefficient of $\lambda^2$  for any $b$, $u\in V$. Then, by the assumption, $V\subset I$. Since $I$ is a $\mathbb{C}[\partial]$-module, $I=R$, i.e., $R$ is simple.

\end{proof}
\begin{remark}
When $(V,\circ)$ is simple, the assumption $V=V\ast V$ for $(V,\ast)$ in Theorem \ref{p1} is not usually satisfied. For example,
for the simple Novikov algebra $(A_2,\circ)$ in Lemma \ref{ll4}, when $2b\in \Delta$, it is obvious that the element $L_{-2b}\notin A_2\ast A_2$.

In addition, there are some natural conditions to make $(V,\ast)$ satisfy the assumption $V=V\ast V$.
For example, $(V,\ast)$ is a Novikov-Jordan algebra with a unit. In special, when $(V,\circ)$ is simple and commutative,
it is easy to see that $(V,\ast)$ satisfies the assumption.
\end{remark}

\begin{corollary}
Given a simple Novikov algebra $(V,\circ)$ and Let $R=\mathbb{C}[\partial]V$ be the corresponding Lie conformal algebra. Then, $\mathbb{C}[\partial]\partial V\subset I$ for any nonzero ideal $I$ of $R$.
\end{corollary}
\begin{proof}
It can be directly obtained from the proof of Theorem \ref{p1}.
\end{proof}

\begin{lemma}\label{l1}
If $(V, \ast)$ is a simple Novikov-Jordan algebra, then $(V,\circ)$ is a simple Novikov algebra.
\end{lemma}
\begin{proof}
If $(V,\circ)$ is not a simple Novikov agebra, then there exists a non-trivial ideal $I$ of $V$. But $(I,\ast)$ is also a non-trivial ideal of
$(V,\ast)$. It contradicts with the simplicity of $(V,\ast)$. Therefore, $(V,\circ)$ is a simple Novikov algebra.
\end{proof}

\begin{proposition}\label{po}
Given a Gel'fand-Dorfman bialgebra $(V,\circ, [\cdot,\cdot])$. If $(V,\ast)$ is a simple Novikov-Jordan algebra, then
the quadratic Lie conformal algebra $R=\mathbb{C}[\partial]V$ corresponding to $(V,\circ, [\cdot,\cdot])$ is simple.

\end{proposition}
\begin{proof}
It can be immediately obtained from Lemma \ref{l1} and Theorem \ref{p1}.
\end{proof}

 Under the assumption in Theorem \ref{p1}, the simplicity of the quadratic Lie conformal algebra is independent with the Lie algebra structure oover $(V,\circ)$. Therefore, Theorem \ref{p1} provides us a strategy to find infinite simple Lie conformal algebras:\\
1. Classify infinite-dimensional simple Novikov algebras.\\
2. Check whether the Novikov-Jordan algebra $(V,\ast)$ satisfies the assumption $V=V\ast V$ for the obtained infinite-dimensional simple Novikov algebra $(V,\circ)$.\\
3. Classify the Lie algebra structures over $(V,\circ)$ if $V=V\ast V$ for $(V,\ast)$.\\

\section{Several classes of new infinite simple Lie conformal algebras}
In this section, according to several classifications of infinite-dimensional simple Novikov algebras  given in \cite{Os1, OZ, X2, X3},
using the strategy mentioned in Section 3, we introduce several infinite simple Lie conformal algebras.
\begin{lemma}\label{ll2}
Assume $A_1=\oplus_{i\geq -1}^{+\infty}\mathbb{C}L_i$ with the product
\begin{eqnarray}\label{eqq5}
L_i \circ L_j=(j+1)L_{i+j},\text{ for all $i$, $j\geq-1$.}
\end{eqnarray}
Then, $(A_1,\circ)$ is a simple Novikov algebra. Moreover, for the Novikov-Jordan algebra $(A_1,\ast)$ where $\ast$ is given by
\begin{eqnarray}
L_i \ast L_j=(i+j+2)L_{i+j},\text{ for all $i$, $j\geq-1$,}
\end{eqnarray} we obtain $A_1=A_1\ast A_1$.
\end{lemma}
\begin{proof}
It is obvious that $(A_1,\circ)$ is a simple Novikov algebra and it also appears in \cite{Os1}.
In addition, it is easy to see that $A_1=A_1\ast A_1$.
\end{proof}

\begin{lemma}\label{ll3}(see Theorem 3.1 in \cite{OZ})
Any Lie algebra structure over $(A_1,\circ)$ is given as follows:
\begin{eqnarray}
[L_i,L_j]=c(i-j)L_{i+j}, \text{ for all $i$, $j\geq-1$,}
\end{eqnarray}
where $c\in \mathbb{C}$.
\end{lemma}

\begin{theorem}
Let $CL_1(c)=\oplus_{i\geq -1}^{+\infty}\mathbb{C}[\partial]L_i$ be a Lie conformal algebra with the $\lambda$-bracket:
\begin{eqnarray}
[{L_i}_\lambda L_j]=((i+1)\partial+(i+j+2)\lambda)L_{i+j}+c(j-i)L_{i+j},
\end{eqnarray}
for all $i$, $j\geq -1$ and $c\in \mathbb{C}$. Then, $CL_1(c)$ is a simple Lie conformal algebra for any
$c\in \mathbb{C}$.
\end{theorem}
\begin{proof}
It can be directly obtained from Theorem \ref{p1}, Lemma \ref{ll2} and Lemma \ref{ll3}.
\end{proof}
\begin{remark}
By the definition of coefficient algebra of a Lie conformal algebra, it is easy to obtain that
$\text{Coeff}(CL_1(c))$ has a basis $\{L_{i,t}|i\geq -1,t\in \mathbb{Z}\}$ over $\mathbb{C}$ and
the Lie bracket is given by
\begin{eqnarray*}
[L_{i,t},L_{j,s}]=((j+1)t-(i+1)s)L_{i+j,t+s-1}+c(j-i)L_{i+j,t+s},
\end{eqnarray*}
for any $i$, $j\geq -1$ and $t$, $s\in \mathbb{Z}$.

When $c=0$, $CL_1(0)$ is introduced in \cite{SY1} which is the graded algebra of general conformal algebra $gc_1$.
\end{remark}

Next, let $\Delta$ be an additive subgroup of $\mathbb{C}$.
\begin{lemma}\label{ll4}
Assume $A_2$ has a basis $\{x_\alpha\}$ where $\alpha$ ranges over $\Delta$, and products are given by
\begin{eqnarray}\label{eqqq5}
x_\alpha \circ x_\beta=(\beta+b)x_{\alpha+\beta},\text{ for all $\alpha$, $\beta\in \Delta$.}
\end{eqnarray}
Then, $(A_2,\circ)$ is a simple Novikov algebra. Moreover, for the Novikov-Jordan algebra $(A_2,\ast)$ where $
\ast$ is given by
\begin{eqnarray}
x_\alpha \ast x_\beta=(\alpha+\beta+2b)x_{\alpha+\beta},\text{ for all $\alpha$, $\beta\in \Delta$,}
\end{eqnarray}
we obtain $A_2=A_2\ast A_2$ when $2b\notin \Delta$.
\end{lemma}
\begin{proof}
It has been shown in \cite{Os1} that $(A_2,\circ)$ is a simple Novikov algebra. In addition, it is easy to see that when $2b\notin \Delta$,
$A_2=A_2\ast A_2$.

\end{proof}
\begin{remark}\label{rrr1}
In fact, when $2b\in \Delta$, $(A_2,\ast)$ is not a simple Novkov-Jordan algebra. For example,
let $J$ have a basis $\{x_\alpha\}$, where $\alpha$ ranges all elements of $\Delta$ except $-2b$. Then, it is easy to check that $J$ is a non-trivial ideal of $(A_2,\ast)$.
\end{remark}
\begin{lemma}\label{ll5}(see Theorem 4.2 in \cite{X1}.)
Any Lie algebra structure over $(A_2,\circ)$  when $b\notin \Delta$ is of the form as follows:
\begin{eqnarray}
[x_\alpha,x_\beta]=\frac{1}{b}(\varphi(\beta)\alpha-\varphi(\alpha)\beta+b(\varphi(\beta)-\varphi(\alpha)))x_{\alpha+\beta}, \text{ $\alpha$, $\beta\in \Delta$,}
\end{eqnarray}
where $\varphi:\Delta\rightarrow \mathbb{C}^+$ is a group homomorphism.
\end{lemma}
\begin{theorem}
Let $CL_2(b,\varphi)=\oplus_{\alpha\in \Delta} \mathbb{C}[\partial]x_\alpha$ be a Lie conformal algebra with the $\lambda$-bracket:
\begin{gather}
[{x_\alpha}_\lambda x_\beta]=((\alpha+b)\partial+(\alpha+\beta+2b)\lambda)x_{\alpha+\beta}\nonumber\\
+\frac{1}{b}(\varphi(\alpha)\beta-\varphi(\beta)\alpha+b(\varphi(\alpha)-\varphi(\beta)))x_{\alpha+\beta},
\end{gather}
for all $\alpha$, $\beta\in \Delta$ and $2b\notin \Delta$, $\varphi:\Delta\rightarrow \mathbb{C}^+$ is a group homomorphism. Then, $CL_2(b,\varphi)$ is a simple Lie conformal algebra.
\end{theorem}
\begin{proof}
Since $\Delta$ is an additive subgroup of $\mathbb{C}$, $2b\notin \Delta$ means $b\notin \Delta$. Then, this theorem can be directly obtained from Theorem \ref{p1}, Lemma \ref{ll4} and Lemma \ref{ll5}.
\end{proof}
\begin{remark}\label{rk1}
It is also easy to obtain that $\text{Coeff}(CL_2(b,\varphi))$ has a basis $\{L_{\alpha,i}|\alpha\in \Delta, i\in \mathbb{Z}\}$ and
the Lie bracket is given by
\begin{eqnarray*}
&&[L_{\alpha,i},L_{\beta,j}]=(i(\beta+b)-j(\alpha+b))x_{\alpha+\beta,i+j-1}\\
&&+\frac{1}{b}(\varphi(\alpha)\beta-\varphi(\beta)\alpha+b(\varphi(\alpha)-\varphi(\beta)))x_{\alpha+\beta,i+j}.
\end{eqnarray*}

When $\Delta=\mathbb{Z}$, $CL_2(b,0)$ is studied in \cite{FCH}.

Moreover, when $2b\in \Delta$, the corresponding Lie conformal algebra may not be simple. For example, according to Remark \ref{rrr1}, when $2b\in \Delta$, $CL_2(b,0)$ has an ideal $B=J\oplus \mathbb{C}[\partial]\partial A_2$ where the direct sum is the sum of vector spaces.

\end{remark}
\begin{lemma}\label{ll6}
Let $A_3$ be a vector space with a basis $\{x_{\alpha,j}|\alpha\in\Delta, j\in \mathbb{N}\}$. For any given constant $b\in \mathbb{C}$, define
an algebraic operation $\circ$ on $A_3$ by
\begin{eqnarray}
x_{\alpha,i}\circ x_{\beta,j}=(\beta+b)x_{\alpha+\beta,i+j}+jx_{\alpha+\beta,i+j-1}, \text{$\alpha$, $\beta\in \Delta$, $i$, $j\in \mathbb{N}$.}
\end{eqnarray}
Then, $(A_3,\circ)$ is a simple Novikov algebra. Moreover, for the Novikov-Jordan algebra $(A_3,\ast)$ where $\ast$ is given by
\begin{eqnarray}\label{eer}
x_{\alpha,i}\ast x_{\beta,j}=(\alpha+\beta+2b)x_{\alpha+\beta,i+j}+(i+j)x_{\alpha+\beta,i+j-1},
\end{eqnarray}
we obtain $A_3=A_3\ast A_3$.
\end{lemma}
\begin{proof}
The result that $(A_3,\circ)$ is a simple Novikov algebra can be referred to Theorem 2.9 in \cite{X2}.
Next, we prove that $A_3=A_3\ast A_3$.

For any $\alpha$, $\beta\in \Delta$ and $\alpha+\beta\neq -2b$, setting $i=0$, $j=0$ in (\ref{eer}), we obtain $x_{\alpha+\beta,0}\in A_3\ast A_3$. Then, letting $i=1$ and $j=0$ and according to $x_{\alpha+\beta,0}\in A_3\ast A_3$,
we obtain $x_{\alpha+\beta,1}\in A_3\ast A_3$ when $\alpha+\beta\neq -2b$. Similarly, we can get $x_{\alpha+\beta,i}\in A_3\ast A_3$ for all
$i\in \mathbb{N}$ when $\alpha+\beta\neq -2b$. If $\alpha+\beta=-2b$, according to (\ref{eer}), we obtain
$x_{-2b,i}\in A_3\ast A_3$ for all
$i\in \mathbb{N}$. Therefore, $A_3=A_3\ast A_3$.

\end{proof}
\begin{remark}
It should be pointed out that the infinite-dimensional Novikov algebra $A$ in \cite{Os1} with a basis $\{L_{\alpha,j}|\alpha\in\Delta, j\in \mathbb{N}\}$ and the products given by
\begin{eqnarray*}
L_{\alpha,i}\circ L_{\beta,j}=(\beta+b)\left(
                                         \begin{array}{c}
                                           i+j \\
                                           i\\
                                         \end{array}
                                       \right)
L_{\alpha+\beta,i+j}+\left(\begin{array}{c}
                         i+j-1 \\
                         i \\
                       \end{array}
                     \right)
L_{\alpha+\beta,i+j-1}.
\end{eqnarray*}  is isomorphic to $A_3$. The isomorphic algebra morphism $\psi$ from $A$ to $A_3$ is as follows:
$\psi(i!L_{\alpha,i})\rightarrow x_{\alpha,i}$, for all $\alpha \in \Delta$ and $i\in \mathbb{N}$.
\end{remark}

\begin{lemma}\label{ll7}(see Theorem 5.1 in \cite{X1})
Any Lie algebra over $A_3$ with $b\notin \Delta$ is of the following form:
\begin{gather}\label{ert}
[x_{\alpha,i},x_{\beta,j}]=\frac{1}{b}((\alpha+b)\varphi(\beta)-(\beta+b)\varphi(\alpha))x_{\alpha+\beta,i+j}\\
+\frac{1}{b}[i(\varphi(\beta)-c(\beta+b))+j(c(\alpha+b)-\varphi(\alpha))]x_{\alpha+\beta,i+j-1}.\nonumber
\end{gather}
for $\alpha$, $\beta\in \Delta$, $i$, $j\in \mathbb{N}$, where $\varphi:\Delta\rightarrow \mathbb{C}^+$ is a group homomorphism and
$c \in \mathbb{C}$ is a constant.

\end{lemma}
\begin{remark}\label{rr1}
In addition,  Xu in \cite{X1} also presented several families of Lie algebras over $(A_3,\circ)$ when $b\in \Delta$.
For example, when $0\neq b\in \Delta$, the Lie algebra structure defined by (\ref{ert}) is also over $(A_3,\circ)$.
When $b=0$, three kinds of Lie algebra structures over $(A_3,\circ)$ are given in Remark 5.2 in \cite{X1}. For example, there is a Lie algebra structure over $(A_3,\circ)$ defined as follows:
\begin{eqnarray}
[x_{\alpha,i},x_{\beta,j}]=\phi(\alpha,\beta)x_{\alpha+\beta,i+j}+(i\varphi(\beta)-j\varphi(\alpha))x_{\alpha+\beta,i+j-1},
\end{eqnarray}
for $\alpha$, $\beta\in \Delta$, $i$, $j\in \mathbb{N}$, where $\varphi:\Delta\rightarrow \mathbb{C}^+$ is a group homomorphism and
$\phi(\cdot,\cdot):\Delta \times \Delta\rightarrow \mathbb{C}$ is a skew-symmetric $\mathbb{Z}$-bilinear form.

\end{remark}
\begin{theorem}
Let $CL_3(b,\varphi,c)=\oplus_{\alpha\in \Delta,i\in \mathbb{N}} \mathbb{C}[\partial]x_{\alpha,i}$ be a Lie conformal algebra with the $\lambda$-bracket:
\begin{gather}
[{x_{\alpha,i}}_\lambda x_{\beta,j}]=\partial((\alpha+b)x_{\alpha+\beta,i+j}+ix_{\alpha+\beta,i+j-1})\nonumber\\
+\lambda((\alpha+\beta+2b)x_{\alpha+\beta,i+j}+(i+j)x_{\alpha+\beta,i+j-1})\nonumber\\
+\frac{1}{b}((\beta+b)\varphi(\alpha)-(\alpha+b)\varphi(\beta))x_{\alpha+\beta,i+j}\nonumber\\
+\frac{1}{b}[j(\varphi(\alpha)-c(\alpha+b))+i(c(\beta+b)-\varphi(\beta))]x_{\alpha+\beta,i+j-1},
\end{gather}
for all $\alpha$, $\beta\in \Delta$, $i$, $j\in \mathbb{N}$, $c \in \mathbb{C}$ is a constant and $b\neq 0$, $\varphi:\Delta\rightarrow \mathbb{C}^+$ is a group homomorphism. Then, $CL_3(b,\varphi,c)$ is a simple Lie conformal algebra.
\end{theorem}
\begin{proof}
It can be directly obtained from Theorem \ref{p1}, Lemma \ref{ll6}, Lemma \ref{ll7} and Remark \ref{rr1}.
\end{proof}
\begin{remark}
According to the definition of coefficient algebra of a Lie conformal algebra, it is easy to see that $\text{Coeff}(CL_3(b,\varphi,c))$ has a
basis $\{x_{\alpha,i,s}|\alpha\in \Delta,s\in \mathbb{N},s\in \mathbb{Z}\}$ and the Lie bracket is given by
\begin{eqnarray*}
[x_{\alpha,i,s},x_{\beta,j,t}]&=&(s(\beta+b)-t(\alpha+b))x_{\alpha+\beta,i+j,t+s-1}+(js-it)x_{\alpha+\beta,i+j-1,t+s-1}\\
&&+\frac{1}{b}((\beta+b)\varphi(\alpha)-(\alpha+b)\varphi(\beta))x_{\alpha+\beta,i+j,t+s}\\
&&+\frac{1}{b}[j(\varphi(\alpha)-c(\alpha+b))+i(c(\beta+b)-\varphi(\beta))]x_{\alpha+\beta,i+j-1,t+s},
\end{eqnarray*}
for any $\alpha$, $\beta\in \Delta$, $i$, $j\in \mathbb{N}$, $b\neq 0$ and $s$, $t\in \mathbb{Z}$.

In fact, according to Remark \ref{rr1}, there are many other Lie algebra structures over $(A_3,\circ)$ when $b\in \Delta$. Thus, following our method, there are many other infinite simple Lie conformal algebras. For example, when
$b=0$, by Remark \ref{rr1}, Lemma \ref{ll6} and Theorem \ref{p1}, there is a simple Lie conformal algebra $CL_3(\phi,\varphi)=\oplus_{\alpha\in \Delta,i\in \mathbb{N}} \mathbb{C}[\partial]x_{\alpha,i}$ with the $\lambda$-bracket as follows:
\begin{eqnarray*}
[{x_{\alpha,i}}_\lambda x_{\beta,j}]&&=\partial(\alpha x_{\alpha+\beta,i+j}+ix_{\alpha+\beta,i+j-1})\\
&&+\lambda((\alpha+\beta)x_{\alpha+\beta,i+j}+(i+j)x_{\alpha+\beta,i+j-1})\\
&&+\phi(\beta,\alpha)x_{\alpha+\beta,i+j}+(j\varphi(\alpha)-i\varphi(\beta))x_{\alpha+\beta,i+j-1},
\end{eqnarray*}
for all $\alpha$, $\beta\in \Delta$, $i$, $j\in \mathbb{N}$, where $\varphi:\Delta\rightarrow \mathbb{C}^+$ is a group homomorphism and
$\phi(\cdot,\cdot):\Delta \times \Delta\rightarrow \mathbb{C}$ is a skew-symmetric $\mathbb{Z}$-bilinear form.
\end{remark}

\end{document}